\documentclass[12pt]{article}

\usepackage{amsmath,amssymb,amsthm, fullpage}
\newcommand{\floor}[1]{\ensuremath{\left\lfloor #1\right\rfloor} }

\newcommand{\size}[1]{\left \vert #1 \right \vert}

\newtheorem{lemma}{Lemma}[section]
\newtheorem{theorem}[lemma]{Theorem}

\newtheorem{corollary}[lemma]{Corollary}

\theoremstyle{definition}
\newtheorem{definition}[lemma]{Definition}

\def\sat{\operatorname{sat}}
\def\ex{\operatorname{ex}}
\def\Msat{\sat_g}    
\def\msat{\sat'_g} 

\def\cC{{\mathcal C}}
\def\cF{{\mathcal F}}
\def\cO{{\mathcal O}}
\def\cP{{\mathcal P}}
\def\cT{{\mathcal T}}
\def\cX{{\mathcal X}}
\def\Gb{{\overline G}}
\def\FR#1#2{\frac{#1}{#2}}
\def\CH#1#2{\binom{#1}{#2}}
\def\NN{\mathbb{N}}
\def\FL#1{\left\lfloor{#1}\right\rfloor}
\def\CL#1{\left\lceil{#1}\right\rceil}
\def\esub{\subseteq}

\title{The Game Saturation Number of a Graph}
\author{
James M. Carraher\thanks{Mathematics Department, University of Nebraska,
Lincoln, NB:  s-jcarrah1@math.unl.edu.  Research supported by NSF grant
DMS 09-14815}\,,
William B. Kinnersley\thanks{Mathematics Department, Ryerson University,
Toronto, ON, Canada, M5B 2K3: wkinners@ryerson.ca}\,,\\
Benjamin Reiniger\thanks{Mathematics Department, University of Illinois,
Urbana, IL: reinige1@illinois.edu}\,,
Douglas B. West\thanks{Mathematics Departments, Zhejiang Normal University
(Jinhua, Zhejiang, China) and University of Illinois (Urbana, IL, U.S.A):
west@math.uiuc.edu.  Research supported by Recruitment Program of Foreign
Experts, 1000 Talent Plan, State Administration of Foreign Experts Affairs,
China.}
}
\date{begun August 12, 2011}

\begin{document}

\date{\today}
\maketitle

\begin{abstract}
Given a family $\cF$ and a host graph $H$, a graph $G\esub H$ is
{\it $\cF$-saturated relative to $H$} if no subgraph of $G$ lies in $\cF$ but
adding any edge from $E(H)-E(G)$ to $G$ creates such a subgraph.  In the
{\it $\cF$-saturation game} on $H$, players {\it Max} and {\it Min} alternately
add edges of $H$ to $G$, avoiding subgraphs in $\cF$, until $G$ becomes
$\cF$-saturated relative to $H$.  They aim to maximize or minimize the
length of the game, respectively; $\Msat(\cF;H)$ denotes the length under
optimal play (when Max starts).

Let $\cO$ denote the family of all odd cycles and $\cT$ the family of
$n$-vertex trees, and write $F$ for $\cF$ when $\cF=\{F\}$.  Our results include
$\Msat(\cO;K_{2k})=k^2$, $\Msat(\cT;K_n)=\CH{n-2}2+1$ for $n\ge6$,
$\Msat(K_{1,3};K_n)=2\FL{n/2}$ for $n\ge8$, 
$\Msat(K_{1,r+1};K_n)=\FR{rn}2-\FR{r^2}8+O(1)$, and
$\size{\Msat(P_4;K_n)-\FR{4n-1}5}\le 1$.  We also determine
$\Msat(P_4;K_{m,n})$; with $m\ge n$, it is $n$ when $n$ is even, $m$ when $n$
is odd and $m$ is even, and $m+\FL{n/2}$ when $mn$ is odd.  Finally, we prove
the lower bound $\Msat(C_4;K_{n,n})\ge\FR1{10.4}n^{13/12}-O(n^{35/36})$.  The
results are very similar when Min plays first, except for the $P_4$-saturation
game on $K_{m,n}$.
\end{abstract}

\baselineskip16pt

\section{Introduction}
The archetypal question in extremal graph theory asks for the maximum number of 
edges in an $n$-vertex graph that does not contain a specified graph $F$ as a
subgraph.  The answer is called the {\em extremal number} of $F$, denoted
$\ex(F;n)$.  The celebrated theorem of Tur\'an~\cite{Tur} gives the answer
when $F$ is the complete graph $K_r$ and determines the largest $n$-vertex
graphs not containing $K_r$ (the {\em size} of a graph is the number of edges).

We consider maximal graphs not containing $F$.  The concept extends to a family
$\cF$ of graphs.  A graph $G$ is {\em $\cF$-saturated} if no subgraph of $G$
belongs to $\cF$ but $G+e$ contains a graph in $\cF$ whenever $e\in E(\Gb)$.
The extremal number $\ex(\cF;n)$ is the maximum size (number of edges) of an
$\cF$-saturated $n$-vertex graph.  (In all notation involving families of
graphs, we write $\cF$ as $F$ when $\cF$ consists of a single graph $F$.)

One may also ask for the minimum size of an $\cF$-saturated $n$-vertex graph;
this is the {\em saturation number} of $\cF$, denoted $\sat(\cF;n)$.
Erd\H{o}s, Hajnal, and Moon~\cite{EHM} initiated the study of graph saturation
by determining $\sat(K_r;n)$.

Generalizing further, a subgraph $G$ of a host graph $H$ is {\em $\cF$-saturated
relative to $H$} if no subgraph of $G$ lies in $\cF$ but adding any edge of
$E(H)-E(G)$ to $G$ completes a subgraph belonging to $\cF$.  The extremal
number and saturation number concern saturation relative to $K_n$, but
saturation has also been studied relative to other graphs.  For example,
Zarankiewicz's Problem involves saturation relative to $K_{n,n}$.  When two
agents have opposing interests in creating a large or a small $\cF$-saturated
graph, we obtain the ``saturation game''.

\begin{definition}
The {\em $\cF$-saturation game} on a host graph $H$ has players {\em Max} and
{\em Min}.  The players jointly construct a subgraph $G$ of $H$ by iteratively
adding one edge of $H$, constrained by $G$ having no subgraph that lies in
$\cF$.  The game ends when $G$ becomes $F$-saturated relative to $H$.  Max aims
to maximize the length of the game, while Min aims to minimize it.  When both
players play optimally, the length of the game is the
{\em game $\cF$-saturation number} of $H$, denoted $\Msat(\cF;H)$ when Max
starts the game and by $\msat(\cF;H)$ when Min starts it.  For clarity and for
consistency with the extremal and saturation numbers, we write the values as
$\Msat(\cF;n)$ and $\msat(\cF;n)$ when playing on $K_n$.
\end{definition}

The saturation game generalizes to any hereditary family of sets.  Let $D$ be a
family of subsets of a set $X$ such that every subset of a member of $D$ also
belongs to $D$.  The saturated subsets are the maximal elements of $D$.  Max
and Min alternately add elements of $X$ to a set that always lies in $D$.  The
game ends when a saturated set is reached, with Max and Min having the same
goals as before.  In the $\cF$-saturation game on $H$, we have $X=E(H)$, and
avoiding subgraphs in $\cF$ defines the hereditary family $D$.

Patk\'os and Vizer~\cite{PV} introduced this general model and studied the
case where $X_n$ is the family of $k$-element subsets of $\{1,\dots,n\}$ and
$D$ is the set of intersecting families of $k$-sets.  View $X_n$ as the
$n$-vertex complete $k$-uniform hypergraph $K_n^{(k)}$.  Letting $M$ be the
forbidden subgraph consisting of two disjoint edges, the game becomes
$\Msat(M;X_n)$.  The Erd\H{o}s--Ko--Rado Theorem~\cite{EKR} then states
$\ex(M;K_n^{(k)})=\CH{n-1}{k-1}\sim \FR1{(k-1)!}n^{k-1}$.  F\"uredi~\cite{Fu}
proved that $\sat(M;K_n^{(k)})\le \FR34 k^2$ when a projective plane of order
$r/2$ exists.  For $k\ge2$, Patk\'os and Vizer~\cite{PV} proved
$\Omega(n^{\FL{k/3}-5})\le \Msat(M;K_n^{(k)})\le O(n^{k-\sqrt k/2})$.

The saturation game is also related to other well-studied graph games.  In a
{\em Maker-Breaker} game, the players Maker and Breaker take turns choosing
edges of a host graph $H$, typically $K_n$.  Maker wins by claiming all of the
edges in a subgraph of $H$ having some specified property ${\cal P}$, and
Breaker wins by preventing this.  For example, Hefetz, Krivelevich,
Stojakovi\'c, and Szab\'o~\cite{HKSS1} studied Maker-Breaker games played on
$K_n$ in which Maker seeks to build non-planar graphs, non-$k$-colorable
graphs, or $K_t$-minors.
%
Several papers have considered the minimum number of turns needed for Maker to
win (see~\cite{FK,HKSS2}).  In this context, Breaker behaves like Max in the
saturation game, making the game last as long as possible.  In the saturation
game both players contribute edges, but here Maker cannot use the edges taken
by Breaker.

In an {\em Avoider-Enforcer} game, again two players alternately choose edges
of a fixed host graph.  Avoider wants to avoid creating any subgraph satisfying
$\cP$; Enforcer wants to force Avoider to build such a subgraph.  Hefetz,
Krivelevich, and Szab\'o~\cite{HKS} introduced such games, establishing general
results and studying the cases where Avoider seeks to avoid spanning trees or
spanning cycles of $H$.
In Avoider-Enforcer games winnable by Enforcer, one may ask how quickly
Enforcer can win.  Here Enforcer behaves like Min in the saturation game,
but again the the moves by Enforcer are not part of Avoider's subgraph
(see~\cite{AJSS,BM,BS2,HKSS3}).
%

The $\cF$-saturation game on $H$ is also related to the {\it $\cF$-free
process} on $H$, equivalent to both players moving randomly.  The length of
the process is the number of moves to reach a graph that is $\cF$-saturated
relative to $H$.  Usually $H=K_n$ (see~\cite{BK,BR,ESW,OT}), but~\cite{BB} is
more general.  For the $C_4$-free process on $K_{n,n}$, the lower bound
of~\cite{BB} specializes to $\Omega(n^{4/3}(2\log n)^{1/3})$.

\bigskip
The saturation game on graphs was introduced by F\" uredi, Reimer, and
Seress~\cite{FRS}; they studied $\Msat(K_3;n)$, calling it ``a variant of
Hajnal's triangle-free game''.  In Hajnal's original ``triangle-free game'',
the players aim only to avoid creating triangles, and the loser is the player
first forced to create one (Ferrara, Jacobson, and Harris~\cite{FJH} considered
the generalization of Hajnal's loser criterion to arbitrary $\cF$ and $G$).
Since the $F$-saturation game always produces an $F$-saturated graph,
$n-1 = \sat(K_3;n) \le \Msat(K_3;n) \le \ex(K_3;n) = \floor{n^2/4}$; hence
$\Msat(K_3;n) \in \Omega(n)\cap O(n^2)$.  F\"uredi et al.~\cite{FRS} proved
$\Msat(K_3;n) \in \Omega(n \lg n)$.  Erd\H{o}s (unpublished) stated
$\Msat(K_3;n) \le n^2/5$.  The correct order of growth remains unknown.

The $P_3$-saturation game was studied by Cranston, Kinnersley, O, and
West~\cite{CKOW}; here $P_k$ denotes the $k$-vertex path.  The subgraphs of $H$
that are $P_3$-saturated relative to $H$ are precisely the maximal matchings in
$H$.  Thus the game $P_3$-saturation number is just the {\em game matching
number}, with $\alpha'_g(G)$ and $\hat\alpha'_g(G)$ denoting the values of the
Max-start and Min-start games since $\alpha'(G)$ denotes the maximum size of a
matching in $G$.  They proved $\alpha'_g(G)\ge \FR23\alpha'(G)$ for every graph
$G$ (with equality for some split graphs) and $\alpha'_g(G)\ge \FR34\alpha'(G)$
when $G$ is a forest (with equality for some trees).  The minimum of
$\alpha'_g(G)$ over $n$-vertex $3$-regular graphs is between $n/3$ and $7n/18$.

We have mentioned bounds on $\alpha'_g$ but not $\hat\alpha'_g$ because the two
parameters never differ by more than $1$ (see \cite{CKOW}).  This does not hold
for $\cF$-saturation in general.  For example, when the host graph is obtained
from a star with $m$ edges by subdividing one edge, the Max-start
$2K_2$-saturation number is $m$, but the Min-start $2K_2$-saturation number is
$2$.  As a less artificial example, we will show that
$|\Msat(P_4;K_{m,n})-\msat(P_4;K_{m,n})|$ can be large, where $K_{m,n}$ is the
complete bipartite graph with part-sizes $m$ and $n$.  In most instances that
we study, the choice of the starting player does not affect the outcome by much.

In Section~\ref{Knsec}, we study the $\cF$-saturation games on $K_n$ for
$\cF\in\{\cO,\cT_n,\{K_{1,r+1}\},\{P_4\}\}$, where $\cO$ is the family of all
odd cycles and $\cT_n$ is the family of $n$-vertex trees.
We first prove $\Msat(\cO;2k)=\msat(\cO;2k)=k^2$, achieving the trivial upper
bound $\ex(\cO;2k)$.  
For $n\ge3$, we prove $\Msat(\cT_n;n)=\msat(\cT_n;n)=\CH{n-2}2+1$, except
$\Msat(\cT_5;5)=6$ and $\msat(\cT_4;4)=3$; note $\ex(\cT_n;n)=\CH{n-1}2$.
Hefetz et al.~\cite{HKNS} have since studied more general versions of both
of these problems.  They studied $\Msat(\cC_k;n)$ and $\Msat(\cX_k)$ where
$\cC_k$ is the family of $k$-connected graphs with $n$ vertices and
$\cX_k$ is the family of non-$k$-colorable graphs.  In both cases, the
value is close to the extremal number.  Lee and Riet~\cite{LR} have generalized
the tree problem in a different direction, studying $\Msat(\cT_k;n)$.

Always $\Msat(K_{1,3};n)$ and $\msat(K_{1,3};n)$ lie in $\{n,n-1\}$.  Except
for $n\in\{2,3,4,7\}$, they are unequal, with $\Msat(K_{1,3};n)$ being the even
value and $\msat(K_{1,3};n)$ being the odd value.  That is,
$\Msat(K_{1,3};n)=2\FL{n/2}$ and $\msat(K_{1,3};n)=2\CL{n/2}-1$ when $n\ge8$.
Note that $\ex(K_{1,3};n)=n$.  For $n>r>2$, it has been checked by computer
that $\Msat(K_{1,r+1};n)=\FL{\FR{rn-1}2}$ when $n\le8$.  We ask whether this
holds for larger $n$; note that $\ex(K_{1,r+1};n)=\FL{\FR{rn}2}$.  K\'aszonyi
and Tuza~\cite{KT} proved $\sat(K_{1,r+1})=\CL{\FR{rn}2-\FR{(r+1)^2}8}$ for
$n\ge 3r/2$.  Lee and Riet~\cite{LR} proved $\Msat(K_{1,r+1};n)\ge (rn/2)-k+1$.

For the $P_4$-saturation game on $K_n$, the value is not asymptotic to the
extremal number.  We prove $\size{\Msat(P_4;n)-\FR{4n-1}5}\le 1$ and
$\size{\msat(P_4;n)-\FR{4n}5}\le .6$, while $\ex(P_4;n)\in\{n,n-1\}$.
Lee and Riet~\cite{LR} proved $n-1\le\Msat(P_5;n)\le n+2$.

In Section~\ref{P4sec}, we study the $P_4$-saturation game on $K_{m,n}$; we may
assume $m\ge n$.  The choice of who starts the game can matter a lot, as do
the parities of $m$ and $n$.  The value of $\Msat(P_4;K_{m,n})$ is $n$ when $n$
is even (equaling $\sat(P_4;K_{m,n})$), $m$ when $m$ is even and $n$ is odd,
and $m+\FL{n/2}$ when $mn$ is odd.  The value of $\msat(P_4;K_{m,n})$ is $m$
when $n\le2$ and $m+\FL{n/2}-\epsilon$ when $n>2$, where $\epsilon=0$ when $mn$
is even and $\epsilon=1$ when $mn$ is odd.

Note that the difference is $m-2$ when $n=2$, and for larger $n$ the difference
is $(n-1)/2$ when $m$ is even and $n$ is odd.  Note also that
$\sat(P_4;K_{m,n})=\min\{m,n\}$, so when $\min\{m,n\}$ is even we obtain an
example where $\Msat(P_4;G)=\sat(P_4,G)$.  We ask whether there are other
interesting examples where $\Msat(\cF;n)$ or $\msat(\cF;n)$ equals
$\sat(\cF;n)$; \cite{FFGJ} provides a survey of saturation numbers as of 2009.

In Section~\ref{C4sec}, we study the $C_4$-saturation game on $K_{n,n}$.
This game is the natural bipartite analogue of the triangle-saturation game on
$K_n$ studied by F\"uredi, Reimer, and Seress~\cite{FRS}.  Every subgraph that
is $C_4$-saturated relative to $K_{n,n}$ is connected, so
$\Msat(C_4;K_{n,n}) = \Omega(n)$.  On the other hand, F\"uredi~\cite{Fur}
proved $\ex(C_4;K_{n,n})= n^{3/2} + O(n^{4/3})$, so
$\Msat(C_4;K_{n,n}) = O(n^{3/2})$.  Our main result is a polynomial improvement
over the natural lower bound: 
$\Msat(C_4;K_{n,n})\ge\FR1{10.4}n^{13/12}-O(n^{35/36})$.

Our results leave many open questions.  The most interesting specific question
is the order of growth of $\Msat(C_4;K_{n,n})$.  One would also like to
understand the conditions under which $\Msat(\cF;n)$, $\Msat(\cF;K_{m,n})$,
or $\Msat(\cF;H)$ does not differ much from the value of the corresponding
Min-start game.

\section{Saturation games on complete graphs}\label{Knsec}

We begin with saturation games on the complete graph $K_n$.
A graph is {\it nontrivial} if it has at least one edge.

\begin{theorem}
$\Msat(\cO;2k)=\msat(\cO;2k)=k^2=\ex(\cO;2k)$.
\end{theorem}
\begin{proof}
An $\mathcal{O}$-saturated graph is a complete bipartite graph.  With $2k$
vertices, the largest has parts of equal size.  It therefore suffices to give
Max a strategy ensuring that after each turn by Max the bipartition of each
nontrivial component is balanced.  Whether Max or Min starts, the first move by
Max ensures this (yielding two isolated edges if Min moves first).

Subsequently, a move by Min can connect two nontrivial components, lie within
a component, connect two isolated vertices, or connect an isolated vertex
to a nontrivial component.  In the last case, since $2k$ is even, Max can
connect another isolated vertex to the same nontrivial component, keeping the
bipartition balanced.  In the other cases, Max can play an edge within a 
nontrivial component or, if they are all complete bipartite (and balanced),
connect two nontrivial components or, if there is just one nontrivial component
and it is balanced, connect two isolated vertices.  If no move is available,
then the game has ended, with $k^2$ moves played.
\end{proof}

The disjoint union of graphs $G$ and $H$ is denoted $G+H$.  The largest
subgraph of $K_n$ containing no spanning tree of $K_n$ is $K_{n-1}+K_1$, with
$\CH{n-1}2$ edges.

\begin{theorem}
If $n\ge3$, then $\Msat(\cT_n;n)=\msat(\cT_n;n)=\CH{n-2}2+1$, except that
$\Msat(\cT_5;5)=6$ and $\msat(\cT_4;4)=3$.
\end{theorem}
\begin{proof}
Every $\cT_n$-saturated subgraph of $K_n$ has the form $K_r+K_{n-r}$ for
some $r$.  Throughout the game there are some number of components, and a move
either joins two components or adds an edge within a component.

If some move by Max leaves at least two nontrivial components, then Min can
maintain this condition after each subsequent move until all vertices are in
nontrivial components, ensuring the upper bound.  Min connects two isolated
vertices if two isolated vertices remain, increasing the number of nontrivial
components to at least $3$, and Max then cannot reduce it below $2$.  When
only one isolated vertex remains, Min connects it to a nontrivial component.

To exceed the upper bound, Max must therefore always leave only one nontrivial
component.  If the move by Max leaves an even number of isolated vertices, then
Min makes an isolated edge, and Max must connect the nontrivial components.
This repeats until Min connects the last two isolated vertices to make a second
nontrivial component that Max cannot absorb.

If the number of isolated vertices is odd after the first move by Max (and the
number of nontrivial components is $1$), then Min works to fix the parity.  If
Max starts, then $n$ is odd.  Min creates $P_3$.  Max now must enlarge the
component to $P_4$ or $K_{1,3}$ to keep the number of isolates odd.  Because
$K_4$ has an even number of edges, Max eventually must reduce the number of
isolates by $1$ or create a second nontrivial component, unless $n=5$.

If Min starts, then Max must create $P_3$, and $n$ is even.  Now Min completes
the triangle, and again Max must reduce the number of isolates by $1$ or
create a second nontrivial component, unless $n=4$.

Max can enforce the lower bound by always leaving only one nontrivial component.
Only when Min connects the last two isolated vertices will a second component
survive.
\end{proof}

Let $kG$ denote the disjoint union of $k$ copies of $G$.

\begin{theorem}
\begin{align*}
\Msat(K_{1,3};n)&=\begin{cases}   n & \text{when }n\in\{3,7\}\cup2\NN-\{2\} \\
                                n-1 & \text{otherwise} \end{cases} \\
\msat(K_{1,3};n)&=\begin{cases}  n-1& \text{when }n\in2\NN-\{4\} \\
                                  n & \text{otherwise} \end{cases}
\end{align*}
\end{theorem}
\begin{proof}
All $K_{1,3}$-saturated graphs are disjoint unions of cycles plus possibly one
isolated vertex or isolated edge (not both).  Hence the only possible outcomes
are $n$ (call this {\it Max wins}) or $n-1$ (call this {\it Min wins}).  Let
$X(n)$ and $Y(n)$ denote the Max-start and Min-start $K_{1,3}$-saturation games
on $K_n$, respectively.

For $n\ge5$, our claim is that the first player wins when $n$ is even and
the second player wins when $n$ is odd, except that Max wins $X(7)$.  After
giving specific strategies for $n\le7$, we provide general strategies for
$n\ge8$ that reduce the problem to the cases $n\in\{5,6\}$.

When $n\le3$, there is no claw, so Min wins when $n\le2$ and Max wins when
$n=3$, no matter who starts.  When $n=4$, Max can create $2K_2$ or $P_4$ to
win, no matter who starts.

In $Y(5)$, Max creates $2K_2$ and can then force $C_5$.  In $X(5)$, Min creates
$2K_2$ and can then close a cycle on the next turn to win.

In $X(6)$, Max completes a triangle if Min makes $P_3$, reducing to $Y(3)$,
which Max wins.  If Min makes $2K_2$, then Max makes $3K_2$ and next $P_6$ to
win.  In $Y(6)$, Min makes $P_4$ on the second move and can then close a
$4$-cycle or $5$-cycle to win.

In $X(7)$, if Min makes $P_3$, then Max closes the $3$-cycle and wins $Y(4)$.
If Min makes $2K_2$, then Max makes $3K_2$.  Whether Min next makes $P_3$ or
$P_4$, Max closes the cycle and wins.  In $Y(7)$, Max makes $P_3$ and will
later win a game played on three or four vertices.

Now assume $n\ge8$.  Let $W$ be the first player when $n$ is even and the
second player when $n$ is odd; we give a winning strategy for $W$.  Player $W$
always leaves the components being one nontrivial path, an even number of
isolated vertices, and some number of cycles, until the number of isolated
vertices is $6$.  By making $P_2$ or $P_3$ in the first round, $W$ initiates
this process.  If the other player $V$ closes the cycle, then $W$ starts a new
path, while if $V$ extends the path or makes an isolated edge the path is left
longer by two edges.  In either case, the number of isolated vertices decreases
by $2$.

When six isolated vertices remain, if $V$ closes the cycle or makes an isolated edge and lets $W$ close the cycle, then the remaining game is the game on six
vertices started by $W$.  If $V$ extends the path, then $W$ closes the cycle to
leave the game on five vertices started by $V$.  We have shown that when $n=6$
the game is won by the first player, and when $n=5$ the game is won by the
second player.
\end{proof}

Because there are only two possible (consecutive) lengths of the
$K_{1,3}$-saturation game on $K_n$, the outcome is determined by who plays
last.  Ferrara, Jacobson, and Harris~\cite{FJH} studied that question
explicitly; in their game the player who moves last wins.  Although their
analysis is similar to ours due to the structure of $K_{1,3}$-saturated graphs,
their result is different: in their game, for $n\ge5$, the first player wins if
and only if $n$ is even, except $n=7$.  In particular, under their criterion
for winning, the number of moves played will always be $n-1$ (except $n=7$).

Our final game on $K_n$ is the $P_4$-saturation game.  Note that during the
game, all components of the built subgraph must be stars or triangles.
Since Max seeks a large ratio of number of edges to number of vertices,
triangles and large stars are beneficial to Max, while small stars are
beneficial to Min.  However, stars with two edges are dangerous for Min,
since Max can turn them into triangles.  This intuition motivates the
strategies for the players.

\begin{theorem}
For $n\ge4$,
\begin{align*}
\FR{4n-6}5\le&\Msat(P_4;n)\le \FR{4n+4}5\\
\FR{4n-3}5\le&\msat(P_4;n)\le \FR{4n+3}5.
\end{align*}
\end{theorem}
\begin{proof}
During the game, let the {\it value} of the current position count a
contribution for each component: $0$ for an isolated vertex or triangle,
$\FR12$ for $P_2$ or $P_3$, and $1$ for a larger star.  The only way to
decrease the value is to turn a copy of $P_3$ into a triangle.  When we speak of
``making'' or ``creating'' a subgraph, we mean producing it as a component of
$G$.

{\it Upper bound: Min strategy.}
While two isolated vertices are available, Min never makes $P_3$, and if Max
makes $P_3$, then Min responds by converting it to $K_{1,3}$.  Otherwise, Min
makes $P_2$, except that when exactly three isolated vertices remain Min
enlarges an existing star with at least two edges (if one exists).  If only one
isolated vertex remains, then Min attaches it to a largest existing star.

With this strategy, each move by Min increases the value by $\FR12$, except 
possibly the last when one isolated vertex remains, or the next-to-last when
exactly three isolated vertices remain.  This strategy ensures that no
triangles are created, unless Max stupidly makes isolated edges and the final
graph is $K_3+\FR{n-3}2P_2$ with $\FR{n+3}2$ edges.  Hence the components are
all stars, and the number of them is $n-m$, where $m$ is the final number of
edges.  Since the strategy also prevents Max from decreasing the value (unless
Max makes isolated edges), the value reaches at least $\FR{m-4}4$, where $m$ is
the final number of edges.  Also the final value is at most the number of
components.  We obtain $\FR{m-4}4\le n-m$, which simplifies to
$m\le \FR{4n+4}5$ (the same computation yields $m\le \FR{4n+3}5$ in the
Min-start game).

{\it Lower bound: Max strategy.}
While an isolated vertex is available, Max never makes $P_2$, except on the
first turn of the Max-start game.  If Min makes $P_2$, then Max turns it into
$P_3$.  If there is no isolated edge, then Max adds an edge to a star with at
least three edges or completes a triangle if no such star exists.

With this strategy, Max never increases the value, except on the first turn of
the Max-start game.  With each Min move increasing it by at most $\FR12$, the
upper bound on the value is $\FR{m+2}4$ (or $\FR{m+1}4$ in the Min-start game).
Also Max ensures that no isolated edge remains, except possibly the initial
move in the Max-start game and an edge joining the last two isolated vertices.
Except for those one or two components, the number of edges in a component is
its number of vertices minus its contribution to the value.  Hence the final
value is at least $n-m-\FR12$ in the Min-start game, or $n-m-1$ in the
Max-start game (an isolated edge contributes $\FR12$ to the value but $1$ to
$n-m$).  We obtain $\FR{m+2}4\ge n-m-1$ in the Max-start game and
$\FR{m+1}4\ge n-m-\FR12$ in the Min-start game, simplifying to $m\ge\FR{4n-6}5$
and $m\ge\FR{4n-3}5$, respectively.
\end{proof}

\section{The $P_4$-saturation game on $K_{m,n}$}\label{P4sec}

Now we study the $P_4$-saturation game on the complete bipartite graph
$K_{m,n}$.  Since $K_{m,n}$ contains no triangles, during the game all
components are stars.  Throughout this section, $X$ and $Y$ are the partite
sets of $K_{m,n}$, with $|X|=m\ge n=|Y|$.  Let an {\it $X$-star} or
{\it $Y$-star} be a star having at least two leaves in $X$ or in $Y$,
respectively.  Recall that $\alpha'(G)$ denotes the maximum size of a matching
in $G$.

\begin{lemma}\label{match}
A graph $G$ that is $P_4$-saturated relative to $K_{m,n}$ has at most
$m+n-\alpha'(G)$ edges.  If it contains both an $X$-star and a $Y$-star (or
an isolated edge), then equality holds.
\end{lemma}
\begin{proof}
Any even cycle contains $P_4$, so $G$ is a forest.  To avoid $P_4$, edges of a
matching must lie in distinct components.  Since $G$ is a forest, $E(G)$ is the
number of vertices minus the number of components, so $|E(G)|\le
m+n-\alpha'(G)$.

A saturated subgraph containing both an $X$-star and a $Y$-star (or an isolated
edge) cannot have isolated vertices.  The components are then nontrivial stars,
so there are $\alpha'(G)$ of them.  Hence there are exactly $m+n-\alpha'(G)$
edges.
\end{proof}

Call a $P_4$-saturated subgraph that contains both an $X$-star and a $Y$-star
a {\it full subgraph}.  A $P_4$-saturated subgraph that is not full has stars
of only one of these types (plus isolated edges, possibly) and thus has $m$ or
$n$ edges.  Hence Max wants to make a full subgraph.  When $\min\{m,n\}$ is
even, Min can prevent this in the Max-start game, and we obtain
$\Msat(P_4;K_{m,n})=\sat(P_4;K_{m,n})$ in that case.  When Max can make
a full subgraph, Lemma~\ref{match} encourages Min to create a large matching.

\begin{theorem}
For $m\ge n\ge1$, the 
$P_4$-saturation numbers of $K_{m,n}$ are given by
\[\sat_g(P_4;K_{m,n})=\begin{cases} n & \text{when $n$ is even},\\
	m & \text{when $n$ is odd and $m$ is even},\\
	m+\floor{\frac{n}{2}} & \text{when $mn$ is odd.} \end{cases} \]
 and
\[\sat'_g(P_4;K_{m,n})=\begin{cases} m & \text{when $n\le 2$ }, \\
m+\floor{\frac{n}{2}} & \text{when $n>2$ and $mn$ is even}, \\
m+\floor{\frac{n}{2}}-1 &\text{when $n>2$ and $mn$ is odd.}\end{cases}\]
\end{theorem}

\begin{proof}
We will consider cases based on who moves first and the parity of $m$ and $n$.
Let $G$ denote the $P_4$-saturated subgraph built during the game.  Again
``making'' a subgraph means producing it as a component of the current graph.

{\it Upper bounds.} We give strategies for Min.
If Max moves first and $m$ or $n$ is even, then Min ensures that only $X$-stars
or $Y$-stars are created, respectively, by immediately extending isolated edges
made by Max to such stars and otherwise enlarging such stars.  The final number
of edges is then $|X|$ or $|Y|$, respectively.

In the other cases, Min just ensures a large matching.  If Max moves first and
$mn$ is odd, or Min moves first and $mn$ is even, then Min makes isolated edges
until a matching of size $\CL{\FR n2}$ is built, later playing any legal move.
By Lemma~\ref{match}, at most $m+\FL{\FR n2}$ moves are played.

If Min moves first and $mn$ is odd, then Min can do slightly better.  If Max
responds to the first move by making an $X$-star or $Y$-star, then the parity
allows Min to ensure that only $X$-stars or $Y$-stars, respectively, will be
played, yielding an outcome of $|X|$ or $|Y|$.  Hence Max must immediately make
another isolated edge.  The moves by Min still yield a matching of size
$\CL{\FR n2}$, and with the extra edge made by Max the bound improves by $1$.

{\it Lower bounds.}  We give strategies for Max.  Since the game cannot leave
an isolated vertex in each part, at least $\min\{m,n\}$ moves are played.  If
an $X$-star is made, then no isolated vertex can be left in $X$, and at least
$m$ moves are made.  In the Max-start game with $n$ odd and $m$ even, Min can
prevent an $X$-star only by leaving only $Y$-stars after each move.  After
$n-1$ moves, Max makes $K_2$ using the last isolated vertex of $Y$, and then
Min is forced to make an $X$-star.  In the Min-start game with $n\le2$, Max
makes an $X$-star immediately.

In the other cases, we may assume $n\ge3$.  Max wants to force a full subgraph
and keep $\alpha'(G)$ small.  In the Min-start game with $mn$ even, Max
responds to the first move by making a $Y$-star if $n$ is even or an
$X$-star if $n$ is odd and $m$ is even.  In the Max-start game with $mn$ odd,
Max makes a $Y$-star on the third move if Min made $K_2$ on the second;
otherwise Max adds to the $X$-star or $Y$-made by Min.

In each of these cases, Max continues enlarging the original $X$-star or
$Y$-star.  If the graph has not become full by the time $X$ or $Y$,
respectively, has only one isolated vertex remaining, then every move has
created a leaf in that part.  By the parity of the size of that part, it is
Max's turn.  Max makes $K_2$, and now Min must make the graph full.

Hence the graph becomes full, so Max takes advantage of Lemma~\ref{match} by
making the initial star large.  Max can play at least $\FL{n/2}$ edges in
the initial star.  Max can play one more such edge on the $n$th move unless Min
has also played edges into stars in the same direction.  Hence
$\alpha'(G)\le n-\CL{n/2}+1$.  By Lemma~\ref{match}, the final number of edges
is $m+n-\alpha'(G)$, which is at least $m+\FL{n/2}$.

For the Min-start game with $mn$ odd, Max cannot do quite as well.  As noted
when discussing upper bounds, if Max makes an $X$-star or $Y$-star on move $2$,
then Min can limit the final number of edges to $m$ or $n$, respectively.
Hence Max makes $K_2$ on move $2$.  If Min makes $K_{1,2}$, then Max makes
the other type of star.  If Min makes $K_2$, then Max makes an $X$-star and
can make a $Y$-star on the next round.

Hence the graph becomes full.  Max subsequently enlarges $Y$-stars until $Y$
has no more isolated vertices.  All moves by Max to that point except the first
two enlarge $Y$-stars, and there is also one such edge among the first four
moves (played by Max or Min).  Letting a maximum matching consist of the first
edge from each component, we thus have
$\alpha'(G)\le n-1-\FL{\FR{n-4}2}=\CL{\FR n2}+1$, so the final number of edges
is at least $m+\FL{\FR n2}-1$.
\end{proof}

\section{The $C_4$-saturation game on $K_{n,n}$}\label{C4sec}

In this section, we study the $C_4$-saturation game on $K_{n,n}$, the natural
bipartite analogue of the F\"uredi-Reimer-Seress problem.  As we have noted,
the trivial lower bound and the result of~\cite{Fur} yield
$\Omega(n)\le \Msat(C_4,K_{n,n}) \le O(n^{3/2})$.  

Our main result is a polynomial improvement of the lower bound:
$\Msat(C_4,K_{n,n}) = \Omega(n^{13/12})$.  We first prove a technical lemma
giving a lower bound on the size of a restricted type of graph that is also
$C_4$-saturated relative to $K_{n,n}$.  Here our interest is the exponent on
$n$; we make no attempt to optimize lower-order terms or the leading
coefficient.

\begin{lemma}\label{C4_sat}
Let $G$ be $C_4$-saturated relative to $K_{n,n}$, and let $c$ and $d$ be
positive constants.  If there exists $S \subseteq V(G)$ with at least
$cn$ vertices in each partite set such that $\size{N(v)\cap S}\le d\sqrt n$ for
all $v\in V(G)$, then $\size{E(G)} \ge an^{13/12}-O(n^{35/36})$, where
$a=\min\{\FR12(\FR{c^2}{2d^2})^{2/3},\FR{c^2}{2d}\}$.
\end{lemma}
\begin{proof}
Let $S_X$ and $S_Y$ be the subsets of $S$ in the two partite sets.  Consider
$x\in S_X$ and $y\in S_Y$ such that $xy\notin E(G)$.  Since $G$ is
$C_4$-saturated relative to $K_{n,n}$, it contains a copy of $P_4$ with
endpoints $x$ and $y$.  Each vertex in $S_X$ has at most $d\sqrt{n}$ neighbors
and hence at least $cn-d\sqrt{n}$ nonneighbors in $S_Y$.  Thus $G$ contains at
least $c^2n^2-cdn^{3/2}$ copies of $P_4$ with endpoints in $S_X$ and $S_Y$;
call such paths {\it essential paths}.  Since each essential path has endpoints
in $S_X$ and $S_Y$, and since no vertex has more than $d\sqrt{n}$ neighbors in
$S$, no edge is the central edge of more than $d^2n$ essential paths.

Let $T$ be the set of vertices of $G$ with degree at least $n^{5/12}$, and let
$b=(\frac{c^2}{2d^2})^{2/3}$.  If $\size{T}\ge bn^{2/3}$, then
$\sum_{v\in T}d(v)\ge bn^{13/12}$, which yields
$\size{E(G)}\ge\frac b2 n^{13/12}$.
Otherwise, let $H$ be the subgraph of $G$ induced by $T$.  Since $H$ is
$C_4$-free, a result of F\"uredi~\cite{Fur} yields
$\size{E(H)} \le (bn^{2/3})^{3/2} + O((bn^{2/3})^{4/3})$, which
simplifies to $\size{E(H)} \le \frac{c^2}{2d^2} n + O(n^{8/9})$.  Multiplying
by $d^2 n$, we conclude that at most $\frac{c^2}2 n^2 + O(n^{17/9})$ essential
paths have central edges in $H$.

Thus at least $\frac{c^2}{2} n^2 - O(n^{17/9})$ essential paths have central
edges incident to a vertex with degree less than $n^{5/12}$.  Each such edge
is the central edge of at most $d n^{11/12}$ essential paths; hence $G$
has at least $\frac{c^2}{2d} n^{13/12} - O(n^{35/36})$ such edges.
\end{proof}

Though the hypotheses of Lemma~\ref{C4_sat} seem technical, they apply whenever
$\Delta(G) \le d\sqrt{n}$.  Hence we obtain a corollary for ordinary saturation
(using $c=1$).

\begin{corollary}\label{cor_C4_sat}
If $G$ is $C_4$-saturated relative to $K_{n,n}$ and $\Delta(G) \le d\sqrt{n}$,
then $\size{E(G)} \ge an^{13/12}-O(n^{35/36})$, where
$a=\min\{\FR12(\FR1{2d^2})^{2/3},\FR1{2d}\}$.
(If $d\le\FR14$, then $a=\FR1{2d}$).
\end{corollary}

Our main result for the $C_4$-saturation game on $K_{n,n}$ follows easily from
Lemma~\ref{C4_sat}.

\begin{theorem}\label{c4sat_lower}
$\Msat(C_4,K_{n,n}) \ge \FR 1{10.4}n^{13/12}-O(n^{35/36})$,
and similarly for $\msat(C_4;K_{n,n})$.
\end{theorem}
\begin{proof}
We provide a strategy for Max that forces the final subgraph of $K_{n,n}$ to
satisfy the hypotheses of Lemma~\ref{C4_sat}.  This strategy governs almost
the first $2n/3$ moves for Max, after which Max plays arbitrarily.

Let $k=\FL{\sqrt{n/3}}-1$.  Max arranges to give degree $k$ to $k$ specified
vertices in each partite set.  Each move by Max makes an isolated vertex
adjacent to a vertex with growing degree; hence it cannot complete a $4$-cycle.
Fewer than $n/3$ vertices are needed by Max in each part, so Min cannot 
exhaust the isolated vertices in either part with fewer than $2n/3$ moves.
After this phase, Max may play any legal move.

In the final subgraph $G$, let $S$ be the set of leaves of the $2k$ specified
stars constructed by Max.  By construction, the stars are disjoint, so $S$
has about $n/3-2\sqrt{n/3}$ vertices in each part.  Moreover, no vertex in $G$
has more than $\sqrt{n/3}$ neighbors in $S$, since each vertex other than the
center of a star is adjacent to at most one leaf of the star.

Thus $G$ satisfies the hypotheses of Lemma~\ref{C4_sat} with $c$ being any
constant less than $1/3$ and $d=\sqrt{1/3}$, from which the claim follows.
\end{proof}

While Theorem~\ref{c4sat_lower} does establish a nontrivial asymptotic lower
bound for $\Msat(C_4;K_{n,n})$, the correct order of growth remains
undetermined.  Lemma~\ref{match} suggests the following question, which would
yield improved lower bounds for $\Msat(C_4;K_{n,n})$: What is the minimum
number of edges in a graph with maximum degree $D$ that is $C_4$-saturated
relative to $K_{n,n}$? 

%
%
%

\end{document}